\documentclass[11pt,a4paper,twoside,english]{article}
\usepackage{graphics}
\usepackage{color}
\usepackage{textcomp}
\usepackage[pdftex]{graphicx}
\usepackage{amsfonts}
\usepackage{listings}          %lista
\usepackage{fancyhdr}
\usepackage{indentfirst}
\usepackage{array} 
\usepackage{newlfont}
\usepackage{microtype} % migliora il riempimento delle righe
\usepackage[english]{babel}
\usepackage[applemac]{inputenc}    %ATTENZIONE A QUESTO COMANDO MAC (di solito compare latin1
\usepackage[T1]{fontenc}
\usepackage{lmodern}
\usepackage{amscd}

\usepackage[fleqn,tbtags]{amsmath}
\usepackage{amssymb,mathrsfs}       %per le formule matematiche fuori dal testo 
\usepackage{dsfont}                            %\mathds{1}
\usepackage{mathtools}
\mathtoolsset{showonlyrefs}			%numera solo le equazioni richiamate con eqref 
\numberwithin{equation}{section}             %numera le equazioni con le sezioni

\usepackage{latexsym}
\usepackage{amsthm}
\usepackage{fancyhdr}
\usepackage{verbatim}

\theoremstyle{definition}                          %stile corsivo
\newtheorem{thm}{Theorem}[section]     %definizione ambiente teorema
\newtheorem{prop}[thm]{Proposition}      %definizione ambiente proposizione
\newtheorem{cor}[thm]{Corollary}            %definizione ambiente corollario
\newtheorem{lem}[thm]{Lemma}             %definizione ambiente lemma

\theoremstyle{definition}               %stile roman
\newtheorem{dfn}[thm]{Definition}%  definizione ambiente definizione
\newtheorem{rem}[thm]{Remark}          

\newtheoremstyle{prova}% hnamei
{3pt}% Space above
{3pt}% Space below
{}% Body font
{}% Indent amounti
{\textbf}% Theorem head font
{.}% hPunctuation after theorem headi
{\newline}% hSpace after theorem headi2
{}% Theorem head spec (can be left empty, meaning `normal')
\theoremstyle{prova}

\def\R{\mathbb R}
\def\N{\mathbb N}

\def\E{\mathbb E}
\def\K{\mathbb K}
\def\P{\mathbb P}
\def\Q{\mathbb Q}

\def\sha{{\cal A}}
\def\shb{{\cal B}}
\def\shc{{\cal C}}
\def\shd{{\cal D}}
\def\shf{{\cal F}}

\def\shl{{\cal L}}
\def\shp{{\cal P}}

\def\1{\mathds{1}}
\def\ra{\rightarrow}

\def\ve{\varepsilon}
\def\be{\begin{equation}}
\def\ee{\end{equation}}
\def\bs{\begin{split}}
\def\es{\end{split}}

\usepackage{authblk}

\author[1]{Carlo Ciccarella}
\author[2]{Francesco Russo}
\affil[1]{EPFL Lausanne\\
  Institut de Math\'ematiques\\ Station 8\\ CH-1015 Lausanne
  (Switzerland)
}

\affil[2]{ENSTA Paris\\
  Institut Polytechnique de Paris\\
Unit\'e de Math\'ematiques appliqu\'ees\\
F-91120 Palaiseau (France)
}

\begin{document}

\date{May 2nd, 2025}

\title {
  $C^{0,1}$-It\^o chain rules
and generalized solutions of parabolic PDEs.}

 % in view of stochastic differential games}

%\begin{document}
\maketitle

{\bf Abstract}

In this paper we first establish an It\^o formula for
a finite quadratic variation process $X$
expanding $f(t,X_t),$ when $f$ is of class $C^2$
in space and is absolutely continuous in time.

Second, via a Fukushima-Dirichlet decomposition
we obtain an explicit chain rule for $f(t,X_t)$,
when $X$ is a continuous semimartingale
and $f$ is a ``quasi-strong solution''
(in the sense of approximation of classical solutions)
of a parabolic PDE.

Those results are applied in a companion paper to establish a
verification theorem in stochastic differential games.

\medskip

\\


%\end{abstract}

%\bigskip
%[\textbf{2010 Math Subject Classification}: \ ] \ {60G15, 60G22, 60H05, 60H07,  60H30, 91G10, 91G80 }
%\medskip

%[\textbf{JEL Classification Codes}:\ ] \  {G10, G11, G12, G13} 

\bigskip

\newpage

%\frontmatter
%\tableofcontents

\section{Introduction}\label{Intro}

In this paper we aim to establish two It\^o chain rules
in view of applications to a stochastic differential game and stochastic control,
see e.g. the companion paper \cite{CR2},
expanding $f(t,X_t)$, where
$f:\R_+ \times \R^d \rightarrow \R$ and $X$ is a stochastic
process.
%%%CARLO VERIFICA
A preliminary version of that paper is
the second part of

The first formula, stated in Proposition
\ref{ItoGen},
is related to $ f \in C_{ac}^{0,2} (\R_+ \times \R^d)$
(see Definition \ref{C02ac}),
when $X$ is a finite quadratic variation process.
%see \cite{fo} and \cite{rv4}
%for the case when $f \in C^{1,2}$. 
The 
second is Corollary \ref{representation}, where
$f\in C^{0,1}(\R_+ \times \R^d)$
is a generalized (quasi-strong) solution of some parabolic PDE
and $X$ is an It\^o process.
This makes use of a weak Dirichlet decomposition
that one could call Fukushima-Dirichlet decomposition,
see Theorem \ref{Rep}.

If $X$ is a semimartingale and $f \in C^{1,2}$,
then
classical It\^o formula tells us that $f(\cdot,X)$ is still a semimartingale.
%$C^{1,2}$-functions of semimartingales are semimartingales themselves. 
If $f$ has lower than $C^{1,2}$-regularity, then $f(\cdot,X)$
can still be a semimartingale, via It\^o-Tanaka type formulae 
involving local time for instance.
Various generalizations appear in the setting  of optimal stopping problems,
where one gets expansion
of $f(\cdot,X)$ that
for   $ f$ having  $C^{1,2}$-regularity within the continuation
($\shc$) and stopping ($\shd$) regions but the spatial derivatives may not be continuous across the
boundary that separates these two sets, see e.g. 
% The early result in \cite{Peskir1} addresses this typical situation in case of one-dimensional continuous semimartingales.
\cite{Peskir2} and
\cite{DeAngelis}.
If $f$ has much lower regularity, for instance
if $f \in C^1(\R^d)$ then
$f(X)$ is not necessarily a semimartingale.
%this is not true anymore
%in general, in 
%and the natural process:
%to consider is not a semimartingale.
If $d = 1$, for instance, if $X$ is a standard Brownian motion then
$f(X)$ is a semimartingale if and only if $f$ is difference of convex functions,
see e.g. \cite{cjps}.
 The first It\^o formulae expanding $f(X)$
when $f\in C^{1}(\R^d)$ and $X$ is a (reversible) semimartingale continuous
was performed by \cite{rv96}, see also \cite{fp} in the case of Brownian
motion under much weaker conditions.
In this case $f(X)$ is generally a Dirichlet process
i.e. the sum of a local martingale and a zero quadratic variation
process, see \cite{FolDir}.
In the jump case, at our knowledge the first
paper in this direction was \cite{erv}.
Those papers expressed the It\^o formula usual term involving 
the second order derivatives in term of covariation.
For instance if $f \in C^1(\R^d)$ and
$S$ is a reversible semimartingale with usual decomposition
$S = M + V$, then
$$f(X) = f(X_0) + \int_0^\cdot \partial_x f(X) dS + \frac{1}{2}
[\partial_x f(X),X)].$$
An alternative approach
was the so called Bouleau-Yor formula, which, in dimension 1
expressed the remainder via local time, see \cite{by}.
Since then an incredible amount of formulae have been produced,
see e.g. \cite{Russo_Vallois_Book} for a (even not complete)
list of references.

Beyond semimartingales, when $f \in C^{1,2}$,
and $X$ is a finite quadratic variation process
the first contribution was done by \cite{fo}
which was followed by a related stochastic calculus
in \cite{rv4}. In the rough case, there were also
many chain rules, see e.g. \cite{gnrv} or
\cite{friz_hairer}, but this goes beyond the scope
of this paper.

The present paper makes use of 
weak Dirichlet processes as a tool.
That notion  was introduced
in \cite{er2}, see also Definition \ref{DWDir}.
If a semimartingale is the sum of a local martingale $M$ plus a finite variation process $A$,
a weak Dirichlet process $X$ (generalizing the notion
of Dirichlet process of \cite{FolDir}), is the sum of a  local  martingale and a process which is adapted and martingale orthogonal, in the sense that the covariation $[A,N]  = 0$ for every continuous local martingale $N$.
The latter decomposition
% of a process (into the sum of a martingale and a martingale orthogonal term)
was also called  Fukushima-Dirichlet decomposition in \cite{rg2}
and subsequent papers.
It turns out that a  $C^{0,1}$-function of  a weak Dirichlet process (with finite quadratic
variation) is again a weak Dirichlet process, see Proposition 3.10 of \cite{rg2}; applications to a verification theorem for stochastic control
were performed in \cite{rg1} in relation with $C^{0,1}$-(so called) strong solution of the Hamilton-Jacobi-Bellman (HJB) PDE.  When $X$ is a c\`adl\`ag process
that result was generalized in Theorem 3.37 in \cite{BandiniRusso_RevisedWeakDir}. 
Fukushima type decomposition was introduced in the path-dependent case by
\cite{DGR2} in the framework of Banach space valued stochastic calculus.
More recently, \cite{BLT} has extended Proposition 3.2 in \cite{rg2} to a
path-dependent setting, by making use of the Dupire derivative.
A subsequent recent paper, 
%of [Tan and alia] constitutes the last generalization of
generalizing a   $C^{0,1}$-type chain rule (which is  indeed a $C^{0,1}$-Fukushima decomposition in our language)
was performed by \cite{BTW},
where the $1$ corresponds to a functional derivative in the space and
also in the measure in the sense of Definition 5.43 of \cite{CarmonaDelarue1}. That derivative is strongly connected with
the Lions-derivative.
    They have a chain rule expanding $V(t,X_t,\shl_{X_t})$ where $X$ is a semimartingale and $\shl_{X_t}$.
    The resulting process in all cases is a weak Dirichlet process. In    \cite{BTW}, the authors obtain a significant application
    to a verification theorem for a McKean–Vlasov optimal control problem when 
    the value function    $V$ is of class $C^{0,1}$.
    That result of course can be applied to  the particular case of classical control problems,
    when the coefficients do not depend on the law. This  is also new,
since,
in the     
    literature, a comparable verification theorem was the one of
    \cite{rg1} which (differently) started with the existence of a $C^{0,1}$-solution
    of Hamilton-Jacobi PDE.

As mentioned earlier, the first chain rule It\^o formula we prove
(Proposition \ref{ItoGen}) holds for functions $f \in C^{0,2}_{ac}$,
(see Definition \ref{C02ac}) and finite quadratic variation processes.
 In particular $f$ is absolutely continuous in time.
 % It provides an It\^o type formula.
 The second It\^o formula (Corollary \ref{representation})
is based on
the Fukushima-Dirichlet decomposition formulated in Theorem \ref{Rep},
holds for $f \in C^{0,1}$,
being a so called {\it quasi-strong solution} 
 (see Definition \ref{strongNU})
of some parabolic PDE
and $X$ being a specific weak Dirichlet process with diffusive
local martingale component and a finite quadratic variation
martingale orthogonal component.

Contrarily to the case of $C^1$-transformations of time-reversible semimartingales,
we drive the attention on the fact that the $C^{0,1}$-decompositions
mentioned earlier do not provide an explicit decomposition
of the martingale orthogonal term. One of the aspects of this paper,
see Theorem \ref{Rep} and Corollary \ref{representation},
consists in formulating  an explicit expression of the martingale orthogonal process.

A quasi-strong solution  to a parabolic PDE problem is defined as the limit
of {\it quasi-strict solutions} (see Definition \ref{strict}) to a
sequence of approximating PDEs,  which are $C^{0,2}_{ac}$-functions.

Potential applications of the present paper
appear each time we need to apply a chain
rule of the type $u(t,S_t)$ where
only a non-regular solution $u$ of an underlying  PDE is available and
$S$ is an underlying stochastic process.
 One is related to
the resolution of the identification
       problem in a forward-backward stochastic differential equations,
       where the solution $(Y,Z)$ in the Brownian case (resp. $(Y,Z,K)$ in the
       jump case), see e.g.       \cite{FuhrmanTessitore} (resp. \cite{BandiniBSDEs}),
       when one is interested
       %in the identification problem, i.e.
       in expressing $Z$ (resp. $Z,K$) as a ``gradient'' of
       a solution of some semilinear Kolmogorov PDE.
       In this case $S$  is a solution
       of an underlying forward equation.

       In the companion paper \cite{CR2},
%(see also \cite{ciccarella-russo}
%Theorem 4.3 and Proposition 3.18)),
we apply Corollary \ref{representation}
to establish a verification theorem in the framework of stochastic control
and stochastic differential games. In particular
one applies the chain rule to quasi-strong solutions of
Hamilton-Jacobi-Bellman (HJB) and Bellman-Isaacs (BI) equations
which in general they are not expected to be neither  $C^{1,2}$
nor $C^{0,2}_{ac}$. 

Indeed, in stochastic control (resp. game theory problems), the driver of the parabolic PDE, which is the HJB (resp. BI equation) is the Hamiltonian,
may be highly non-regular. In our formulation, it corresponds to the function $h$ in \eqref{parabolic}. Therefore, a priori, we cannot expect much regularity also from the solution to the PDE. If we now replace
% the driver
$h$ (resp. the terminal function $g$)
by a sequence of smoother (in space) functions, converging to the
% Hamiltonian
$h$ (resp. $g$),
we can hope that the existence of solutions to this sequence of PDEs is somehow easier to prove.
This is the main idea behind these two notions of solutions. 
The notion of ``quasi-strong'' (resp. ``quasi-strict'')
solution of Definition \ref{strongNU} (resp. \ref{strict}) extends the one of ``strong'' (resp. ``strict'') solution of \cite{rg2}, Definition 4.2  (resp. Definition 4.1).
In particular, the quasi-strict solution is
% in the sense that they do not require
% continuity in the time variable of the coefficients and
% the solution to be
not of class
$C^1$ in time.
In \cite{rg2}, the coefficients of the PDE differential operators are
supposed to be continuous.
In this work, we also drop any time regularity assumption on the coefficients of the parabolic operator \eqref{operator}.

In \cite{rg2}, the convergence of the approximating functions to the
Hamiltonian (characterizing the notion of ``strong solution'') is supposed to be uniform in time and space, therefore implying that the Hamiltonian has to be continuous in time.
In our formulation of quasi-strong solution, we replace uniform
convergence in  time  by an $L^1$-convergence, thus allowing for Hamiltonians which are not continuous in time
as well as the coefficients of the PDE.
An example of existence of quasi-strict solution is shown in Lemma
\ref{StrictEx}.  We also show that the notion
of quasi-strong solution is related to that of a mild solution.

Our results are organized as follows. 
In Section \ref{sec:FQV} we establish Proposition
\ref{ItoGen}, which states the It\^o formula for $C^{0,2}_{ac}$ functions
of finite quadratic variation processes.
Section \ref{SSolutionConcept} is devoted to various types of solutions of PDEs, i.e. quasi-strict, which generalizes
the classical, mild and quasi-strong.
Finally Section \ref{sec:C01Chain} is devoted to the explicit chain rule type of subclasses of $C^{0,1}$-functions, see  Theorem \ref{Rep}
and Corollary \ref{representation}.

 \section{Complements on stochastic calculus for finite quadratic variation processes}

 \label{sec:FQV}

First, we recall some basic definitions and fix some notations, then we recall the definition of the Fukushima-Dirichlet decomposition that plays a key role in this paper.
%the proof of Theorem \ref{Verification3}. 

\subsection{Analytical preliminaries}

 \label{SPrelim}

In this section, $ 0 \leq t < T <\infty$ will be fixed. The definition and conventions of this
section will be in force for the whole paper.
By convention a continuous function $\phi(s)$, $s \in [t,T]$, is extended to the whole line setting 

$\phi(s) = \left\{ 
\begin{array}{l l}
\phi(t) \ \ \ &s \leq t ,\\
\phi(T) \ \ \ &s\geq T.
\end{array}\right. $

If $E$ is a (subset of) finite-dimensional  space,
%i.e.  a complete metric space,
 $C^{0} (E) $ denotes the space of all continuous functions $f: E \rightarrow \R$.
It is a Fr\'echet space equipped with the topology of 
the uniform convergence of $f$  on each compact.
If
% $(E,d)$ is a metric space and
$f : E \ra \R$
is a uniformly continuous function,  
we denote by $\gamma(f, \cdot)$ the modulus of continuity of $f$.

If $k$ be a non-negative integer,
$C^k(\R^d)$ denotes the space of the functions such that all the derivatives up to order $k$ exist and are continuous. It is again 
a Fr\'echet space equipped with the topology of 
the uniform convergence of $f$ and all its derivatives on each compact.
Let $I$ be a compact real interval.   $C^{1,2} ( I \times \R^d)$ (respectively $C^{0,1} ( I \times \R^d)$), is the space of continuous functions
$f: \ I \times \R^d \ra \R $,  
% $ \ \ \ (s,x) \mapsto f(s,x)$
such that  $ \partial_s f, \partial_x f, \partial^2_{xx} f  $ (respectively $ \partial_x f $)
are well-defined  and  continuous. This space is also a  Fr\'echet space.
In general $\R^d$-elements will be considered as column vector,
with the exception of $\partial_x f$ which will by default
a row vector.

\begin{dfn} \label{C02ac}
  %Let $m$ be a non-negative integer. 
  $C^{0,2}_{ac}(I \times \R^d) $ will be  the linear space of continuous
  functions $f: I \times \R^d \ra \R  $ such that
  the following holds.
\begin{enumerate}
 \item $f(s,\cdot) \in C^2 (\R^d)$ for   all $s \in I$.
 \item For any $x \in \R^d $ the function $s\mapsto f(s,x) $ is absolutely continuous and
   % $\partial_s f(s,\cdot)$ is continuous
   for almost all $s \in I$.
    $x \mapsto \partial_s f(s,x)$ is continuous.
\item  For any compact  $K \subset \R^d $,  $\sup_{x\in K} |\partial_s f(\cdot, x) | \in L^1 (I).$
\item
  Let $g$ be any second order space derivative of $f$.
  \begin{enumerate}
\item  $g$ is continuous with respect to the space variable $x$ varying on each compact $K$, uniformly with respect to $s \in I;$
  \item {\it for every $x \in \R^d$, $s \mapsto g(s,x)$ is a.e. continuous.}
\end{enumerate}
\end{enumerate}
Suppose that $f$ fulfills all previous items except 4.(b) which is replaced by

\medskip
4.(b) Bis: {\it for every $x \in \R^d$ , $s \mapsto g(s,x)$ has at most
  countable discontinuities.}

\medskip
In this case $f$ will be said to belong 
to
$C^{0,2}_{ac, count}(I \times \R^d). $ 
\end{dfn}

In the sequel $(\Omega, \shf, (\shf_s)_{s\geq t}, \P)$ will be a given stochastic basis satisfying the usual conditions. $(W_s)_{s\geq t} $ will denote a classical
$(\shf_s)_{s\geq t} $-Brownian motion with values in $\R^d$. 
A sequence of processes $(X^n_s)_{s\geq t} $ will be said to converge u.c.p.
if the convergence holds in probability uniformly on compact intervals.

The space  $ C_\shf(I \times \Omega; \R) $
of  all continuous $\R^d$-valued
$(\shf_s)_{s\in [t,T]}$-progressively measurable  processes
is a Fr\'echet space, if 
equipped  with a metric related the u.c.p. convergence.

\subsection{Weak Dirichlet processes and Fukushima Dirichlet decomposition}

\begin{dfn} \label{DInt}
  Let $(X_s)_{s \in [t,T]} = (X^1_s, \cdots, X^d_s)_{s \in [t,T]}$ be a vector of continuous processes and $(Y_s)_{s \in [t,T]} = (Y^1_s, \cdots, Y^d_s)_{s \in [t,T]}  $ be vector process with integrable paths. We define
  $$I^-(\ve, Y, dX)(s) =  \int_t^s Y_r \cdot \frac{X_{r + \ve} - X_r}{ \ve} dr,
  $$ where $\cdot$ denotes the inner product in $\R^d.$
If $d = 1$ then, of course, we omit the $\cdot$.
  We also say that the \textbf{(vector) forward integral} of $X$ with respect to $Y$ exists if

\begin{enumerate}
\item $ \lim_{\ve \ra 0} I^-(\ve, Y, dX)(s)  $ exists in probability for any $s \in [t,T]$,
\item the previous limit admits a continuous version.
\end{enumerate}
That integral limit process, if exist, is denoted by $ \int_t^s Y_r d^-X_r, s \in [t,T]$.
\end{dfn}

\begin{dfn} \label{DCov}
  Let $(X_s)_{s \in [t,T]} = (X^1_s, \cdots, X^d_s)_{s \in [t,T]}$ and $(Y_s)_{s \in [t,T]} = (Y^1_s, \cdots, Y^d_s)_{s \in [t,T]}  $ be two vector of continuous processes. We define
 % $$ C(\ve, Y, X)(s) =  \frac{1}{\ve}  \int_t^s (X_{r+ \ve} - X_r) \cdot (Y_{r + \ve} - Y_r) dr,$(\varepsilon_n)$ $$
  \begin{equation} \label{epsCov}
[X,Y]^\ve(s) =
\frac{1}{\ve}  \int_t^s (X_{r+ \ve} - X_r) \cdot (Y_{r + \ve} - Y_r) dr, s \in [t,T]
    \end{equation}
    % where $\cdot$ denotes the inner product in $\R^d$
and we say that
  the \textbf{(matrix) covariation or bracket} of $X$ and $Y$ exists if
\begin{enumerate}
\item $\lim_{\ve \ra 0} [X,Y]^\ve(s)  $ exists in probability for any $s \in [t,T]$,
\item the previous limit admits a continuous version.
\end{enumerate}
That  matrix-valued limit process, if it exists, is denoted by $ [X,Y]_s$.
\end{dfn}
If $X$ is a real process, $[X,X]$ will be also called {\bf quadratic variation} of $X$.
The forward integral, the covariations and the finite quadratic variation processes were introduced in  \cite{rv} and the first steps were performed in \cite{rv2}, 
see also the recent monograph \cite{Russo_Vallois_Book}.
\begin{rem} \label{Ito=Forw}
\begin{enumerate}
\item  If $X$ is a continuous $(\shf_s)_{s\geq t}$-semimartingale and $Y$ is an $(\shf_s)_{s\geq t}$-progressively measurable c\`agl\`ad  process
  (resp. an  $(\shf_s)_{s\geq t}$-semimartingale) then
  $\int_s^\cdot Y d^- X$  (resp. $[Y,X]$)
  coincides with $\int_s^\cdot Y dX$  (resp. the usual covariation).
  In particular if $X$ is bounded variation process then
  $\int_s^\cdot Y d^- X$ is the usual Lebesgue integral $\int_s^\cdot Y dX,$
    see Proposition 2.1 of \cite{rv2}.
  \item If $X$ is a continuous $\R^d$-valued $(\shf_s)_{s\geq t}$-semimartingale and $Y$ is an
$\R^d$-valued
    $(\shf_s)_{s\geq t}$-progressively measurable c\`agl\`ad  process
  then
  $\int_s^\cdot Y d X$  will denote the vector It\^o integral $\int_s^\cdot Y \cdot dX.$ 
  \end{enumerate}
  \end{rem}

  We define below a natural extension of the notion of Dirichlet processes.
  Those processes were defined by H. F\"{o}llmer, see Chapter 14 in
  % \cite{er},
  \cite{FolDir}, see also \cite{Russo_Vallois_Book}.
  %\cite{rg2}.
We remind that an $(\shf_s)_{s \ge t}$-progressively measurable process $X$ is 
called Dirichlet if it is the sum of $(\shf_s)_{s \ge t}$-local martingale $M$
and a zero quadratic variation process $A$, i.e. such that $[A,A] = 0$.

\begin{dfn} \label{DWDir}
\begin{enumerate}
\item 
An  $(\shf_s)_{s\in [t,T]}$-continuous progressively measurable process real-valued process $X$ is called $(\shf_s)_{s\in [t,T]}$-\textbf{weak Dirichlet process} if
\begin{equation}\label{Dec2}
 X= M + A,
\end{equation}
where
\begin{enumerate}
 \item $(M_s)_{s\in [t,T]}$ is a continuous $(\shf_s)_{s\in [t,T]}$-local martingale, 
\item $A_0 = 0 $ and $A$ is an $(\shf_s)_{s\in [t,T]}$-\textbf{martingale orthogonal process}, that is  $[N,A] = 0 $ for every $(\shf_s)_{s\in [t,T]}$-continuous local martingale $N$. 
\end{enumerate}
\item A continuous $\R^d$-valued process $ X = (X^1, \cdots, X^d)$ is said to be a $(\shf_s)_{s\in [t,T]}$-weak Dirichlet (vector)
(resp. $(\shf_s)_{s\in [t,T]}$-\textbf{martingale orthogonal} process),
if each component is $(\shf_s)_{s\in [t,T]}$-weak Dirichlet
(resp. {martingale orthogonal} process).
\end{enumerate}
\end{dfn}
\eqref{Dec2} was also called {\bf Fukushima-Dirichlet} decomposition, see \cite{rg2}.
\begin{rem}

\begin{enumerate}
\item The decomposition \eqref{Dec2} is unique, see Proposition 15.1 in
  \cite{Russo_Vallois_Book}.
  \item In the sequel, when self-explanatory, we will omit the filtration $(\shf_s)_{s\in [t,T]}$.
  \item A zero quadratic variation process is a martingale orthogonal process,
    by Kunita-Watanabe theorem, see Proposition 4.7 item 5. in  \cite{Russo_Vallois_Book}. 
 \item   Usual properties of weak Dirichlet processes are given in \cite{er}, \cite{er2}, \cite{rg2},
and in Chapter 15 of \cite{Russo_Vallois_Book}.
\end{enumerate}

%\cite{Rus05}.
\end{rem}

\subsection{The It\^o formula for finite quadratic variation
processes}

We state now our first It\^o type formula
which generalizes the classical F\"ollmer-It\^o type formula, see
\cite{FolDir} and \cite{rv2} in the framework of
calculus via regularization.

\begin{prop}\label{ItoGen}
  Let $f \in C^{0,2}_{ac}([t,T] \times \R^d ) $ and $(X_s)_{s \in [t,T]}$ be a
  continuous
  $\R^d$-valued process that
$[X,X]$ exists.

Suppose that one of the two conditions below is fulfilled.
\begin{enumerate} 
\item $f \in C^{0,2}_{ac,count}([t,T] \times \R^d ). $ 
\item The mutual covariations $[X^i,X^j]$
  are absolutely continuous, where $[X,X] = ([X^i,X^j])$.
  \end{enumerate}
  Then the following It\^{o} formula holds.
\begin{align}\label{ItoGen2}
  f(s, X_s)&= f(t,X_t) + \int_t^s  \partial_r f(r,X_r) dr +
               \int_t^s \partial_x f (r,X_r)  d^-X_r \\
&\qquad + \frac{1}{2} \sum_{i,j=1}^d \int_t^s \partial_{ij}^2 f(r,X_r) d[X^i,X^j]_r.
\end{align}
\end{prop}

\begin{rem} \label{RNotations}
For shortness we denote the last  sum term  in
\eqref{ItoGen2} 
by $\int_t^s  {\partial^2_{xx} f }(u,X_u) d[X,X]_u.  $
%(resp. $\int_t^s  {\partial_{x} f }(u,X_u) d^-X_u).$
\end{rem}

  Before the proof we want to say a few words about the preceding assumptions. 
\begin{rem}\label{remItogen}
  One consequence of $f \in C^{0,2}_{ac,count}$
  (resp. $f \in C^{0,2}_{ac}$)
% \begin{comment}
%   \begin{enumerate}
% \item
%   %a sufficient condition for the validity of Item 2. is the following:
%  {\it $t \mapsto g(t, x)$ is continuous for every $t \in D^c$
% for every $x \in \R^d$.}
% \item For every continuous function $k:[t,T] \rightarrow \R^d$
%   {\it $t \mapsto g(t, k(t))$ is continuous for every $t \in D^c$.}
% \end{enumerate}
% \end{comment}
  is the following:
  %The hypothesis 4. (b) Bis implies that, 
  for every  second derivative (in space) $g$ of $f$
  there
  is a countable (resp. Lebesgue null) subset $D$   of $[t,T]$ such that
 %The hypothesis 4. (b) Bis implies that, every  second derivative $g$ of $f$
  %in space fulfills the following. There
  %is a countable subset $D$  of $[t,T]$ such that
  %a sufficient condition for the validity of Item 2. is the following:
  the following holds.
  \begin{enumerate}
\item
  %a sufficient condition for the validity of Item 2. is the following:
{\it $t \mapsto g(t, x)$ is continuous for every $t \in D^c$
for every $x \in \R^d$.}
\item For every continuous function $k:[t,T] \rightarrow \R^d$, we have 
  {\it $t \mapsto g(t, k(t))$ is continuous for every $t \in D^c$.}
\end{enumerate}

%  {\it $t \mapsto g(t, x)$ is continuous for every $t \in D^c$
%for every $x \in \R^d$.}

\bigskip

We explain below the reasons.
Of course we can substitute $\R^d$ in previous sentence
with a fixed compact $K$ of $\R^d$.

Concerning item 1.,
let now $\ve > 0 $. Since $K$ is compact, it is included in a finite union of balls  $B(x_i, \ve)$, where $\{ x_i: \ i = 1 , \dots , N    \} $ are elements of
$K$. By item 4(b) Bis (resp.   item 4(b))
 there is a countable (resp. zero Lebesgue measure) set $D \subset [t,T]$ such that $s \mapsto g(s, x_i)$, $i = \{1, \dots, N \} $ is continuous outside $D$.
Let $s_0 \notin D$ and $x \in K$. We show that
$g(\cdot,x)$ is continuous  in $s_0$.

There is $i \in \{1, \dots, N   \} $ such that $|x - x_i | < \ve$ so that,
for $s \in [t,T]$,
we have
\begin{align}
|g(s, x) - g(s_0, x)  | &\leq | g(s, x) - g(s,x_i)    |  + |g(s, x_i) - g(s_0, x_i) | + | g(s_0, x_i) - g(s_0, x)   |  \\
&\leq \gamma (g (s, \cdot)_{\vert K}; \ve) + |g(s, x_i) - g(s_0, x_i) | +  \gamma (g_{\vert _K}(s_0, \cdot); \ve),
\end{align}
where  $\gamma$ is the modulus of continuity.
%of
%some uniformly continuous function $k$.
Taking on both sides the $\limsup_{s \ra s_0}$, we get
\begin{align}
\limsup_{s \ra s_0}  |g(s, x) - g(s_0, x)  |   \leq 2  \sup_{s \in [t,T]}
\gamma (g(s, \cdot); \ve).  
\end{align}
We now take the limit when $\ve \ra 0$ getting 
\begin{align}
\limsup_{s \ra s_0}  |g(s, x) - g(s_0, x)  | = 0.    
\end{align}
%for all $x \in K$, $s_0 \notin D$.

Concerning item 2., let $s_0 \in D^c$ and $(s_n)$ be a sequence converging to $s_0$.
We have
\begin{eqnarray*}
\vert g(s_n, k(s_n)) - g(s_0,k(s_0)) \vert &\le&
\vert g(s_n, k(s_n)) - g(s_n,k(s_0)) \vert
+ \vert g(s_n, k(s_0)) - g(s_0,k(s_0)) \vert \\ &\le&
\sup_{s\in [t,T]} \gamma(g(s,\cdot)_{\vert K}; \vert k(s_n) - k(s_0)\vert)
  +  \vert g(s_n, k(s_0)) - g(s_0,k(s_0)) \vert.
  \end{eqnarray*}                                 
The first term goes to zero because
of item 4. (a)
of Definition \ref{C02ac}
  and the second term converges to zero by item 4 (b) Bis (resp.
  4. (b)) of Definition \ref{C02ac}.
\end{rem}

\begin{proof}[Proof of Proposition \ref{ItoGen}] 
  For simplicity of notations we write the proof in the case $d=1$.
%We write the proof in the case when $f \in C^{0,2}_{ac}([t,T] \times \R^d ) $ 
%(i.e. under item 1.), since the other case follows similarly.

  For $r \in [t,T]$ we have
\begin{equation}
\begin{aligned}
f(r + \ve, X_{r + \ve}) = & f(r,X_r) + f(r+\ve, X_{r+\ve})- f(r,X_{r+\ve}) \\
  & \qquad + \partial_x f(r,X_r)(X_{r+\ve} - X_r) + \partial^2_{xx} f(r,X_r) \frac{(X_{r+ \ve} - X_r)^2}{2} \\
&\qquad + \frac{1}{2}\int_0^1 da \ [\partial^2_{xx} f(X_r + a(X_{r + \ve} - X_r)) - \partial_{xx}^2 f(X_r)]  (X_{r+ \ve} - X_r)^2.  
\end{aligned}
\end{equation}
Integrating from $t$ to $s$ and dividing by $\ve$, we get
\begin{align}
& I_0(s,\ve) :=  \frac{1}{\ve} \int_t^s f(r + \ve, X_{r + \ve}) - f(r,X_r) dr =  I_1(s,\ve) +  I_2(s,\ve) +  I_3(s,\ve) +  I_4(s,\ve), 
\end{align}
where
\begin{align}
& I_1(s,\ve) =  \frac{1}{\ve} \int_t^s f(r+\ve, X_{r+\ve})- f(r,X_{r+\ve}) dr ,\\
& I_2(s,\ve) = \ \frac{1}{\ve} \int_t^s \partial_x f(r,X_r)(X_{r+\ve} - X_r) dr, \\
& I_3(s,\ve) =   \frac{1}{\ve} \int_t^s \partial^2_{xx} f(r,X_r) \frac{(X_{r+ \ve} - X_r)^2}{2}  dr, \\
& I_4(s,\ve) =   \frac{1}{2 \ve} \int_t^s \int_0^1 da \ [\partial^2_{xx} f(r, X_r + a(X_{r + \ve} - X_r)) - \partial_{xx}^2 f(r,X_r)] \ (X_{r+ \ve} - X_r)^2   dr. 
\end{align}
Of course, as $\ve \ra 0$, since $f$ is jointly continuous, 
\begin{align}
I_0(s,\ve) = \frac{1}{\ve} \int_s^{s+ \ve} f(r,X_r)  dr - \frac{1}{\ve} 
\int_t^{t+\ve} f(r, X_r)   dr  \ra f(s, X_s) - f(t,X_t), \ \ s \in [t,T], \ \ \text{u.c.p.} 
\end{align} 
Let $s \in [t,T]$. We have % ,by item 2. of Definition \ref{C02ac},
\begin{equation*}
  I_1(s,\ve) =  \frac{1}{\ve}  \int_t^s dr \ \int_r^{r+ \ve} da \ \partial_a f(a, X_{r+\ve} ) =  I_{1,1}(s,\ve)  + I_{1,2}(s, \ve),
\end{equation*}
  % \begin{comment}
%   \frac{1}{\ve} \int_t^{s+ \ve} da \ \int_{(a - \ve)\vee t}^{a \wedge s} dr \ \partial_a f(a, \xi( a, \ve)), 
% \end{align} 
% where $\xi(a, \ve) \in [\min_{r \in [a, a + \ve]} X_r , \max_{r \in [a, a + \ve]} X_r ]$ and
% \begin{align}
% {\textit Leb} \ ([\min_{r \in [a, a + \ve]} X_r , \max_{r \in [a, a + \ve]} X_r ]) \leq \gamma (X; \ve).
% \end{align}
% In particular $|X_a - X_r |  < \gamma (X;\ve)$ for all $r \in [a , a + \ve]$, so fixing $\omega \in \Omega$
% \begin{align}
%  I_1(s,\ve) =   I_{1,1}(s,\ve)  + I_{1,2}(s, \ve),
% \end{align}
% \end{comment}
where
\begin{align}
I_{1,1}(s,\ve) &=  \frac{1}{\ve} \int_t^{s+ \ve} da \ \int_{(a - \ve)\vee t }^{a \wedge s} dr \ \partial_a f(a, X_a); \\
I_{1,2}(s, \ve) &= \frac{1}{\ve} \int_t^{s+ \ve} da \ \int_{(a - \ve)\vee t}^{a \wedge s} dr \ [\ \partial_a f(a,X_{s+\varepsilon}) - \partial_a f(a, X_a ) \ ].
\end{align}
It is straightforward that
\begin{align}
I_{1,1}(s,\ve) &= \int_t^{s+ \ve} da \ \partial_a f(a,X_a) \frac{(a \wedge s) - [(a - \ve)\vee t]}{ \ve}, 
\end{align}
converges to
\begin{align}
I_{1,1}(s,\ve) &\ra \int_t^s da \ \partial_a f(a,X_a)  \ a.s.,
\end{align}
% (even uniformly)
because of the Lebesgue dominated convergence theorem,
as $\ve \ra 0$.

Now we investigate  the convergence of the other term
\begin{equation*}
  I_{1,2}(s, \ve) = \int_t^{s+ \ve} J_{1,2}(a,\ve) \ da,
\end{equation*}
with
\begin{align}
J_{1,2}(a,\ve) =   \frac{1}{\ve} \int_{(a - \ve)\vee t}^{a \wedge s} dr \ [ \partial_a f(a, X_{r+ \ve})) -  \partial_a f(a, X_a )]. 
\end{align}
We fix $\omega \in \Omega$.
We prove the convergence for fixed $a \in [t,s]$
We have
\begin{equation*}
J_{1,2}(a,\ve)  =   \frac{1}{\ve} \int_{(a - \ve)\vee t}^{a \wedge s} dr \ [ \partial_a f(a, \xi(a, \ve)) -  \partial_a f(a, X_a )],
\end{equation*}
where $\xi(a, \ve) \in [\min_{r \in [a, a + \ve]} X_r , \max_{r \in [a, a + \ve]} X_r ]$ and
\begin{align}
{\textit Leb} \ ([\min_{r \in [a, a + \ve]} X_r , \max_{r \in [a, a + \ve]} X_r ]) \leq \gamma (X; \ve).
\end{align}
In particular $|X_a - X_r |  < \gamma (X;\ve)$ for all $r \in [a , a + \ve]$.
Consequently, for almost all $a \in [t, s ], $ it holds $ |J_{1,2}(a,\ve)|
\leq \gamma (\partial_a f(a,\cdot); \gamma (X; \ve)) ) \ra 0  $ as $\ve \ra 0$ since $\partial_a f(a, \cdot)$ is uniformly continuous on each compact. % by item 2 of Definition \ref{C02ac}.

Moreover
\begin{equation} \label{eJ12}
  |J_{1,2}(a,\ve) | \leq  2 \sup_{x \in \K_t}
  %\ Supp(X_r  ; \ r \in [t,T])}
  |\partial_a f(a, x) |,
\end{equation}
  where
$   \K_t:= [\min_{r \in [t,T]} X_r , \max_{r \in  t, T]} X_r ].$ 
Since the right-hand side of \eqref{eJ12} is integrable by item 3. of Definition \ref{C02ac} on $[t,T]$, then 
\begin{align}
\lim_{\ve \ra 0} I_{1,2}(s , \ve) = 0,
\end{align}
 by Lebesgue dominated convergence theorem. This implies that
\begin{align}
\lim_{\ve\ra 0}I_1(s, \ve) = \int_t^s da \ \partial_a f(a,X_a),
\end{align}
a.s. and therefore in probability for any $s \in [t,T]$.

We want to prove now that 
\begin{equation} \label{EI3}
  I_3(s,\ve) \rightarrow_{\varepsilon \rightarrow 0}
  % \frac{1}{\ve} \int_t^s \partial^2_{xx} f(r,X_r) \frac{(X_{r+ \ve} - X_r)^2}{2}  dr
    \int_t^s \partial^2_{xx} f(r,X_r)   d[X,X]_r, \forall s \in [t,T],
\end{equation}
in probability.
% to $  \int_t^s \partial^2_{xx} f(r,X_r)   d[X,X]_r $, $\forall s \in [t,T]$.
We will actually prove that
the convergence holds u.c.p.
By Lemma 3.1 in \cite{rv2}
%\ref{Dini}, as $\ve\ra 0$, 
\begin{equation*} 
 \frac{1}{\ve} \int_t^s (X_{r+\ve}-X_r)^2 dr \ra [X,X]_s, \ \ s\in [t,T],\ \ \text{u.c.p.}
\end{equation*}
Let us now  $(\ve_n)$ be a sequence converging to zero.
There exists a subsequence still denoted by $(\ve_n)$  such that, setting
\begin{align}
  \mu_{\ve_n}(s)  = [X,X]^{\ve_n}(s),
%  \frac{1}{\ve_n} \int_t^s (X_{r+\ve_n}-X_r)\ (X_{r+\ve_n} - X_r)dr , 
\end{align}
where the right-hand side was defined in \eqref{epsCov},
then 
\begin{align}
\lim_{\ve\ra 0 } [\sup_{s \in [t,T]} |\mu_{\ve_n}(s) - [X,X]_s  |] = 0  ,\ \ \text{a.s.}
\end{align}
Let $M$ be a $P$-null set such that for $\omega \notin M$ the sequence of real functions $\mu_{\ve_n}(\omega, \cdot) \ra [X,X]_\cdot (\omega)$ uniformly on $[t,T]$, which
implies that $d \mu_{\ve_n}(\omega, \cdot) \Rightarrow d[X,X]_\cdot (\omega)$. 
Suppose the validity of item 1. (resp. item 2.)
Given a continuous function $k: [t,T] \mapsto \R^d$,
using  Remark \ref{remItogen},
the set of discontinuities of
$r \mapsto \partial^{2}_{xx} f(r, k(r))$
is countable 
(resp.  has zero Lebesgue measure).
This implies that (for the fixed $\omega \notin M$), for every $s \in [t,T]$,
\begin{equation*}
  %\label{eq:E1}
  \int_t^s   \partial^2_{xx} f(r, X_r(\omega)) d\mu_{\ve_n}(\omega, r) \ra  \int_t^s
  \partial^2_{xx} f(r, X_r(\omega)) d[X,X]_r(\omega), 
\end{equation*}
making use of Portemanteau theorem.
The convergence can be shown to be uniform in $s$
by decomposing the second derivative(s) in difference
of positive and negative part,
via Dini's theorem
This implies in particular
\begin{align} 
  I_3(\cdot,\ve)
  % =   \frac{1}{\ve} \int_t^\cdot \partial^2_{xx} f(r,X_r) \frac{(X_{r+ \ve} - X_r)^2}{2}  dr
  \ra \int_t^\cdot   \partial^2_{xx} f(r, X_r) d[X,X](r) \ \ \text{uniformly a.s.},
\end{align}
consequently u.c.p. and finally \eqref{EI3}.
Now, let us investigate the convergence of the fourth integral $I_4$.
We have
\begin{equation}
  I_4(s,\varepsilon) \le   \frac{1}{2} 
  \int_t^T \eta(\ve,r) d[X,X]^\varepsilon(r),
  % (X_{r+ \ve} - X_r)^2   dr,
 \end{equation}
  where  $[X,X]^\varepsilon$ was introduced in \eqref{epsCov},
   and
  \begin{eqnarray*}
    \eta(\ve,r) &=& \int_0^1 \frac{1}{2} \int_t^s   \vert
                    \partial^2_{xx} f(r, X_r + a(X_{r + \ve} - X_r))
                  - \partial_{xx}^2 f(r,X_r)\vert  da \\ &\le& 
                        T \sup_{r \in [t,T]}
         \gamma(\partial^2_{xx}f_{K}(r, \cdot);\gamma(X,\ve)).
          \end{eqnarray*}
  Since $[X,X]^\ve$ converges u.c.p. to $[X,X]$ we obtain that
  $I_4(\cdot, \ve) \rightarrow 0$ u.c.p.

 At this point the limit in probability of $I_2(s, \ve)$ is forced to exist and to be continuous because the other terms are. This will be of course $\int_t^s f(r, X_r) d^- X_r$ and the proof is complete.  
\end{proof}

In the Proposition \ref{FukuDir} below, item 1. is
the so called Fukushima-Dirichlet decomposition
and items 2. connects it to the It\^o chain rule
stated in Proposition \ref{ItoGen}.

\begin{prop}\label{FukuDir}
  Suppose $(X_s)_{s\in [t,T]}$ is an $(\shf_s)_{s\in[t,T]}$-weak Dirichlet continuous process
  such that $[X,X]$ exists
  (resp. exists and it is
  absolutely continuous).
  Let  $X = M+ A $ be its decomposition, componentwise.

  There is a continuous linear map
   $\shb^X : C^{0,1}([t,T] \times \R^d) \rightarrow C_\shf ([t,T] \times \Omega; \R^d)$ such that the following holds.
  
  \begin{enumerate}
\item
For every $f \in C^{0,1}([t,T] \times \R^d)$ we have
\begin{equation}  \label{eq:WD}
  f(s, X_s)= f(t,X_t) + \int_t^s \partial_x f (r, X_r) d M_r + \shb^X(f)_s, \  s \geq t,
  %- \shb^X (u)_t ,
\end{equation}
and $\shb^X (f)_s$ is a martingale orthogonal process.
In particular, 
for every $f \in C^{0,1}([t,T]  \times \R^d)$, $(f(s,X_s), s\in[t,T])$ is a $(\shf_s)_{s\in[t,T]}$-weak Dirichlet process with martingale part
$ M^f_s = \int_t^s \partial_x f (r, X_r) d M_r$, $s\in [t,T]$.
\item Suppose $f \in C^{0,1}$.
  If $f\in C^{0,2}_{ac, count}([t,T]  \times \R^d)$ (resp.
  $f\in C^{0,2}_{ac}([t,T]  \times \R^d)$),
we have
   \begin{align} \label{eq:IF}
     \shb^X(f)_s&=& \int_t^s \partial_r f(r,X_r)dr + \int_t^s \partial^{2}_{xx} f(r,X_r)d[M,M]_r 
                    +  \int_t^s \partial_x f(r,X_r)  d^- A_r, s \ge t. \nonumber\\
     &&
   \end{align}
   In particular $  \int_0^\cdot \partial_x f(r,X_r) \cdot d^- A_r  $
 exists and it is  a martingale orthogonal process.  
% \item Suppose that $f\in C^{0,1}$.
% If $f\in C^{0,2}_{ac, count}$ (resp. $f\in C^{0,2}_{ac}$),
%   then $  \int_0^s \partial_x f(r,X_r)d^- A_r  $ is a martingale orthogonal process.  
\end{enumerate}
\end{prop}
\begin{proof} 
Item 1. was the object of item (c) of 
Proposition 3.10 in \cite{rg2}.
Concerning item 2., by the It\^{o} formula \eqref{ItoGen2},
we have
% \eqref{eq:IF}, taking into account
first \eqref{eq:WD} with  \eqref{eq:IF}, taking into account
\begin{equation*} 
\int_t^s \partial_x f(r,X_r)  d^- X_r =
 \int_t^s \partial_x f(r,X_r)dM_r +
 \int_t^s \partial_x f(r,X_r) d^- A_r, s \ge t,
 \end{equation*}
see Remark \ref{Ito=Forw}. 
The last assertion  follows since $\shb^X(f)$ is a martingale orthogonal
process.
% is proved  the same way as item (b) in Proposition 3.10 in \cite{rg2},
%  whereas, instead of using the It\^o formula for finite
%  quadratic variation process, see e.g. Proposition 2.4 in \cite{rg2}.
%Item 3. 
 %we make use of 
 % whereas the verification of item 2. follows by a direct application of
 %the It\^{o} formula \eqref{ItoGen2} in Proposition \ref{ItoGen}.
\end{proof}
We remark that item 2. will not be used in the sequel and it has a
separate interest. To prove  Theorem \ref{Rep}, crucial will be item 1.
of Proposition \ref{FukuDir} and Proposition \ref{ItoGen}.

\section{Concept of solution of a parabolic PDE}

\label{SSolutionConcept}
\subsection{The framework}

Proposition \ref{FukuDir} above has
some significant implications when $f$ is a (generalized) solution
of 
a parabolic partial differential equation.

In the sequel $L(\R^m, \R^d)$ will stand for the linear space of
$d \times m$ real matrices.
$b: [0,T] \times \R^d \rightarrow \R^d $,
$\sigma:[0,T] \times \R^m \rightarrow L(\R^m, \R^d) $ will be
 %      locally bounded\
Borel functions, such that for all $x \in \R^d$,
$  \vert \sigma \vert(\cdot,x) + \vert b\vert(\cdot,x)\vert
\in L^1([0,T]$.
%i.e.
% for all $K$ compact, $\sup_{r \in [0,T], x\in K} (|b(r,x)| +|\sigma(r,x)|) < \infty$
 Let the linear parabolic operator
given formally by
\begin{align}\label{operator}
\shl_0 u(s,x) = \partial_s u (s,x) +  \partial_x u(s,x) \cdot b(s,x) + \frac{1}{2} Tr [\sigma^\top(s,x) \partial^{2}_{xx}u(s,x) \sigma(s,x)].
\end{align}
In this section we will fix Borel functions
$g: \R^d \rightarrow \R$
and $h :[t,T] \times \R^d \ra \R $ such that
\begin{equation} \label{hCond}
  \int_t^T \sup_{x \in K} |h(s,x)| ds < \infty,
\end{equation}
  for every compact $K \subset \R^d$.

We will   consider the inhomogeneous backward parabolic problem
\begin{equation} \label{parabolic}
 \left\{
\begin{array}{l l}
\shl_0 u(s,x)= h(s,x), & s\in [t,T], \ x\in \R^d,\\
u(T,x)= g(x), & x\in \R^d.
\end{array}\right. 
\end{equation}

\subsection{Quasi-strict solutions}

The  definition below generalizes the notion of strict solution, that
appears in \cite{rg1}, Definition 4.1.
%%% Menzionare l'applicazione al controllo
% Below we introduce the notion of quasi strict solution.
For $s \in [t,T],$  we set
\begin{equation} \label{eq:As}
\sha_s f(x) \doteqdot  \partial_x f(x) \cdot b(s,x)
+ \frac{1}{2} Tr [\sigma^\top(s,x) \partial^{2}_{xx}u(s,x) \sigma(s,x)].
\end{equation}

\begin{dfn}\label{strict}
 We say that
$u : [t,T] \times \R^d \ra \R$, $ u \in C^{0,2}_{ac}([t,T]\times \R^d )$ is a \textit{quasi-strict solution} to the backward Cauchy problem \eqref{parabolic} if 
\begin{equation} \label{Estrict}
  u(s,x)= g(x) -  \int_s^T h(r,x) dr + \int_s^T (\sha_r u)(r,x) dr, \
  \forall x \in \R^d.
\end{equation}
 
\end{dfn}
\begin{rem}  \label{Rstrict}
  In this case, for every $x \in \R^d$,
 \begin{equation}\label{Estrict1}
   \partial_s u(s,x) =  h(s,x) - (\sha_s u)(s,x),  \quad ds \ {\textit a.e.}
  \end{equation}
  where $\partial_s u$  stands for the distributional derivative of $u$.
  \end{rem}

The notion of quasi-strict solution allows to consider
(somehow classical) solutions of \eqref{parabolic}
even though $h, \sigma, b$ are not continuous in time.

\begin{dfn} \label{Non-deg}
We say that  $\sigma$ to be {\bf non-degenerate}
  we mean that  there is a a constant $c > 0$ such that for
  all $(s,x) \in [0,T] \times \R^d$ and $\xi \in \R^d$ we have
\begin{equation} \label{not-deg}
  \ \xi^\top \sigma(s,x)^\top \sigma(s,x) \xi \ge c \vert \xi\vert^2.
\end{equation}
\end{dfn}

  \begin{lem}\label{StrictEx}
  Suppose the following.
  \begin{enumerate}
  \item   $\sigma, b$  are H\"older continuous in space (uniformly in time) and $\sigma$ is non-degenerate. Moreover $\sigma, b: [t,T] \times \R^d \rightarrow \R$ a.e. continuous  in time for all $x \in \R^d$. 
 \item $h$ has polynomial growth in space (uniformly in time).
Moreover $h: [t,T] \times \R^d \rightarrow \R$ a.e. continuous  in time for all $x \in \R^d$.
\item   $g$ is continuous with polynomial growth. 
  \item   For every compact $K$ of $\R^d$, $h$ is H\"older continuous
  in space (uniformly in time).
  \end{enumerate}

 Then there is a quasi-strict solution of \eqref{parabolic}.
 \end{lem}
 
%  \end{equation}
  \begin{rem} \label{RLucertini}
    \begin{enumerate}
    \item The proof of Lemma \ref{StrictEx}
     makes essentially use 
of Proposition 4.2
in \cite{lucertini2022optimal}.
    \item Indeed, the context of \cite{lucertini2022optimal}
      applies to our case, $B = 0$.
     Therein, one shows existence of  a so called strong Lie solution $u$,
 % belongs to $C^{0,2}_{ac}$ and it is a quasi strict solution.
     %Indeed here $g$ is smooth.
%   \item We drive the attention on the fact, that $u$ fulfills
which fulfills    \eqref{Estrict}. In particular one gets
\begin{equation} \label{C02acDiff}
 \partial_s u(s,x) = h(s,x) + (\sha_s u)(s,x), \forall x \in \R^d,
 \quad ds \ {\text a.e.}
\end{equation}
\item However, the notion of quasi-strict solution
  requires also $u$ to belong to $C^{0,2}_{ac}(t,T] \times \R^d)$.
   The spatial first and second order spatial derivatives of
   $u$  in  \cite{lucertini2022optimal} are
  (even H\"older) continuous in time, for fixed $x$.
Therefore $u$ 
     fulfills items 1.,  3. and 4. of Definition
  %   by the fact that the equation 
     \ref{C02ac}.
  
  Moreover \eqref{Estrict} and \eqref{C02acDiff} imply
  that $u$ also
    fulfills item 2. of Definition
  %   by the fact that the equation 
     \ref{C02ac}. 
%     $\partial_su$.

  \end{enumerate}
\end{rem}

\begin{rem}
  If  $h \in C^{0,\gamma}([t,T] \times \R^d ) $ and $g$
  is H\"older continuous, the result appears in Theorem
  5.1.9 of \cite{Lunardi} and in this case we even have a strict solution. 
 
\end{rem}

\subsection{Quasi-strong solutions}

The definition below is a relaxation of the notion
of strong solution defined for instance in \cite{rg2}, Definition 4.2.,
which is based on approximation of strict (classical) solutions.
The notion of quasi-strong solution is motivated
by the fact that the coefficients $\sigma$ and
$b$ are not supposed to be continuous in time.

\begin{dfn}\label{strongNU}
  $u: [t,T]\times \R^d \rightarrow \R$ is a \textit{quasi-strong solution}
(with  \textit{approximating sequence} $u_n$)
to the backward Cauchy problem \eqref{parabolic} if there exists
a sequence  $u_n \in C^{0,2}_{ac}([t,T] \times \R^d)$
% and two
and a sequence
%$(g_n): \R^d \rightarrow \R$,
%\subset C^0(\R^d),  $
$h_n:[t,T] \times \R^d \rightarrow \R$
such that 
$ \int_t^T \sup_{x \in K} |h_n(s,x)| ds < \infty,  $ for every compact $K \subset \R^d$ 
%$ (h_n) \subset C^{0}([t,T]\times {\R^d})$
realizing the following.

\begin{enumerate}
\item $\forall n \in \N$, $u_n$ is a quasi-strict solution of the equation
%%% Consider later the notion of quasi good solution 
\begin{equation}  
%\left\{
%\begin{array}{l l}
  \shl_0 u_n(s,x)= h_n(s,x), s \in [t,T], \ x\in \R^d.
  %\\
%u_n(T,x)= g_n(x), & x\in \R^d. 
%\end{array}\right. 
\end{equation}

%\item $\partial_x u_n \in L^1_{loc}([t,T] \times \R^d)$, 
\item For each compact $K \subset \R^d $ 
 \begin{equation}
   \left\{
 \begin{array}{l l}
  \sup_{(s,x) \in [t,T] \times K } | u_n - u |(s,x)  \ra 0,             \\ 
\sup_{x \in K} |h_n -h|(\cdot,x) \ra 0  \qquad \text{in } L^1([t,T]). 
\end{array}\right. 
\end{equation}
\item $u(T, \cdot) = g$.
\end{enumerate}
\end{dfn}
\begin{rem} \label{RStrongNU}
  According to previous definition any
  quasi-strong solution is continuous.

  \end{rem}

\subsection{Mild solutions}
  
We define now the notion of a mild solution.
We suppose here for simplicity that $\sigma$ and $b$
are of polynomial growth.
%i.e.
 %  such that $|b(s, x)| + |\sigma(s,x)| \leq C (1 + |x|), \ \forall (s,x) \in [0,T] \times \R^d$.
  Suppose moreover that the SDE
 \begin{equation}\label{A1}
 \left\{ 
\begin{array}{l }
dX_s = b(s, X_s) \ ds + \sigma (s, X_s) \ dW_s, \\
X_t= x,
\end{array}
\right. 
\end{equation}
admits existence and uniqueness in law for every initial condition for any $ 0 \leq t \leq s \leq T$ and any $x\in \R^d$. 
Let $\check \Omega:= C([t,T])$ be the canonical path space equipped with its Borel $\sigma$-algebra and
  let   $X = (X_s)_{s\in [t,T]}$ be the canonical process.

  For the notion of mild solution  we introduce the associated inhomogeneous semigroup
  $(\shp_{t,s})_{0 \leq t \leq s \leq T}$.
  %for our purpose we use a probabilistic formulation.
 For this
 we denote $(\P^{t,x}), 0 \leq t  \leq T, x\in \R^d$ the
associated Markov canonical class, see e.g. \cite{barrasso-russo3, barrasso-russoAF},
with associated expectation  $(\E^{t,x}_s)_{0 \leq t \leq s \leq T}$.
$\P^{t,x}$ will be the solution of the (unique) solution in law of \eqref{A1}.
In the sequel  $P(t,s;x, dz)$ will denote the (marginal) law of $X_s$
under $\P^{t,x}$.

For $f : \R^d \ra \R^d $ continuous with polynomial growth we set 
\begin{align} \label{eq:mild}
 \shp_{t,s}f(x) =  \E^{t,x}f(X_s).
\end{align}
%where $X$ is the canonical process.
  $(\shp_{t,s})_{0 \leq t \leq s \leq T}$  is then the time-inhomogeneous semigroup associated with the generator \eqref{eq:As}. 

\begin{dfn}\label{Mild}
  Suppose $g$ and $h$
  with polynomial growth.
  A function $u :[t,T] \times \R^d $ is said to be a \textit{mild} solution to \eqref{parabolic} if 
\begin{equation}  \label{eq:mildBis}
  u(s, x) = \shp_{s,T} g(x) + \int_s^T \shp_{s,r} h(r,x) dr, \
  (s,x) \in [t,T] \times \R^d.
\end{equation}
\end{dfn}
  
\begin{rem} \label{RMoments}
  \begin{enumerate}
    \item
By the usual techniques of stochastic calculus, see e.g. Burkholder-David-Gundy and Jensen's inequalities we can get
\begin{equation}
  %\label{E3}
\sup_{t\in[0,T]} \E^{t,x}\left(\sup_{t \le s \le T}  \vert X_s \vert^p\right) \le c \vert x \vert^p, \ \forall x \in \R^d.
\end{equation}
\item Taking into account item 1.,  \eqref{eq:mild} is well-defined.
  Moreover, 
  since $h$ and $g$ have polynomial growth
  a mild solution has necessarily polynomial growth.
\item A mild solution is not a priori a continuous function.
   Consider  $h= 0$ and 
   and $g$ bounded and not continuous. The solution cannot
   be continuous at $t= T$.
\end{enumerate}
\end{rem}

\subsection{Relation between the different concepts of solutions}

Below, for simplicity of the formulation  we set $t= 0$
as initial time of the interval $[t,T]$ in the PDEs.
We will need a supplementary time $t$ in the proof.
\begin{prop}\label{StrMild}
  Let us suppose $\sigma, b, h$
  with polynomial growth.
A quasi-strict solution $u$ of \eqref{parabolic}
with polynomial growth is also a mild solution.

\end{prop}

\begin{proof}
  Let $u$ be a quasi-strict solution of \eqref{parabolic}
  with polynomial growth. In particular, for every $t \in [0,T]$,
  $u \in C^{0,2}_{ac}([t,T] \times \R^d)$.
We fix $t \in [0,T]$ and we apply
  the It\^{o} type formula \eqref{ItoGen2}
in Proposition \ref{ItoGen}, taking into account Remark \ref{Ito=Forw}.
  Clearly  the canonical process $X$ is a solution
  % of \eqref{A1}
of $dX_s = b(s, X_s)  \ ds+ \sigma (s, X_s) \ dW_s,$
  under $\P^{t,x}$.
  Explicitly, for $s \in [t,T]$, we get
\begin{align} \label{eq:usx}
  u(s,X_s) &= u(t,x) + \int_t^s \left
             (\partial_r u(r, X_r) dr + \partial_x u(r, X_r) [b(r, X_r) dr
          + \sigma(r, X_r) dW_r] \right) \nonumber  \\
           &\qquad  + \frac{1}{2} \int_t^s Tr [\sigma^\top(r,X_r) \partial^2_{xx}u(r, X_r) \sigma(r,X_r)] dr \nonumber \\
           &= u(t,x) + \int_t^s  \partial_x u(r, X_r)  \sigma(r,X_r) dW_r +
             \int_t^s \partial_ru(r,X_r) dr + \int_t^s (\sha_r u)(r,X_r) dr  \nonumber \\
           &= u(t,x) + M_s + 
             \int_t^s h(r,X_r) dr,
 \end{align}
 where
 \begin{equation}
  M_s = \int_t^s  \partial_x u(r, X_r)  \sigma(r,X_r) dW_r, \quad s \in [t,T]. 
  \end{equation}
The last equality in \eqref{eq:usx} holds because
$u $ is a quasi-strict solution of \eqref{parabolic}
so  by \eqref{Estrict},
% $u(T,\cdot) = g$
% and
for all $x \in \R^d$, $u(\cdot, x),$ 
 is absolutely continuous and by Remark \ref{Rstrict}, \eqref{Estrict1} holds.
% $\partial_r u(r,x) = - \sha_r u(r,x) + h(r,x), \ dr$ a.e. 
 Since $u$ and $h$ have polynomial growth,
 Remark \ref{RMoments} implies that
 the local martingale vanishing at zero $M$
 is such that
 $\sup_{s\in [t,T]} \vert M_s \vert$ is an integrable r.v.,
 which implies that it is a martingale.
Being $u$ a solution of \eqref{parabolic}
   $u(T,\cdot) = g$ 
   and so \eqref{eq:usx} implies
\begin{equation} \label{eq:usxBis}
  g(X_T) = u(t,x) +  M_T +  \int_t^T h(r,X_r) dr.
   \end{equation} 
   Taking the expectation under $\P^{t,x}$, being $M$
   is a martingale, therefore it has a null expectation, gives
\begin{equation}
 \shp_{t,T}g(x) = u(t,x) + 
     \int_t^T \shp_{t,r} h(r,x) dr,
   \end{equation}
   taking into account \eqref{eq:mild} and Fubini's theorem.
   Finally $u$ is a mild solution
   of \eqref{parabolic}.

\end{proof}

Now we prove that, under certain assumptions, a mild solution is a quasi-strong solution. Again we will take $t= 0$ as initial time of the PDE, for simplicity.

%  Vedere se non si puo' rendere il problema
%%% non degenere. Temo solo introducendo la nozione di
%%good solution o fare dei commenti appropriati.
\begin{prop} \label{P321}
Assume item 1., 2. and 3. of Lemma \ref{StrictEx} and that the restriction of $h$ to each compact is
 continuous in space, uniformly in $s \in [0,T]$.

Then a mild solution is a quasi-strong  solution. 
\end{prop}
\begin{rem} \label{R321}
\begin{itemize}
\item The assumptions of Proposition \ref{P321} is therefore
  slightly weaker than those of Lemma \ref{StrictEx}.
\item
   Of course $h$ does not need to by
  bicontinuous in previous statement.
\item The hypothesis on $h$ to be continuous in space
  is not necessary if we suppose some regularity
  for the transition semigroup
  $ P(t,r; x, da),$ but it goes beyond the scope of the paper.
\end{itemize}
\end{rem}
\begin{proof}[Proof of Proposition \ref{P321}]

  Let $(\P^{t,x})$ be the family of probability measures introduced
  after \eqref{A1}.

  We write 
  \begin{equation*}
    %\label{B1}
u: = v^0 + v, 
\end{equation*}
 where
\begin{align}
v^0 (t,x) &= \E^{t,x} (g (X_T)),\\
v(t,x) &= \E^{t,x} \left(\int_t^T h(r, X_r) dr \right). 
\end{align}
Without restriction of generality we can suppose $g = 0$, since the general case
can be treated similarly.
We  prove that there exists a sequence of quasi-strict solutions $v_n $ of \eqref{parabolic}
with $g = 0$, according to Definition \ref{strongNU}.
We proceed in two steps:
\begin{enumerate}
 \item truncation,
\item regularization.
\end{enumerate}

1. \underline{Truncation} \\

This step consists in reducing the problem
to the case when $h$ is  bounded with compact support.
Let $h$ be as in the assumption
and $h_n(t,x) = h(t,x) \chi(x)_{[-n, n]^d},$ where $\chi_{[-n,n]^d}$ is a smooth function bounded by $1$ which support is included in $[-(n+1), n+1]^d$ and equal to $1$
on $[-n,n]^d$.
%The sequence $(h_n)$ are compatible with 2. of Definition \ref{strongNU}.
We define 
\begin{equation}\label{B2}
v_n(t,x) = \E^{t,x} \left(\int_t^T h_n(r,X_r) dr\right).
\end{equation}
Obviously $v_n$ is a mild solution of \eqref{parabolic}
with $h$ replaced with $h_n$.

%By Lemma \ref{StrictEx} $v_n$ is a quasi strict solution of \eqref{parabolic}.
By hypothesis, there exists $p \ge 1$ and a  positive constant $C$  such that
$|h(r, x)|\leq C(1+|x|^p)$,
for all $(r,x) \in [0,T] \times \R^d$.
For every $(t,x) \in [0,T] \times \R^d$, taking $X= (X_s)_{s \ge t}$ as the canonical
process solution of \eqref{A1}, with $X_t = x$.
Its law is of course $\P^{t,x}$ and we have
\begin{align}
|v_n(t,x) - v(t,x) | & \leq \int_t^T \E^{t,x}  \vert h(r,X_r) \vert  1_{\{X_r \notin [-n, n ]^d \}}   dr \\
& \leq  C \int_t^T   \E^{t,x} ((1+ |X_r|^p)  1_{\{X_r \notin [-n, n ]^d \}})  dr \\
&  \leq 2C  \int_t^T  (\E^{t,x} ((1+ |X_r|)^{2p}))^\frac{1}{2}   \P^{t,x} \{|X_r| \geq n \}^\frac{1}{2}  dr.
\end{align}
By Chebyshev this is smaller than 
\begin{align*}
 \int_t^T  2C (\E^{t,x} ((1+ |X_r|)^{2p}))^\frac{1}{2}    \frac{\E^{t,x} (|X_r|)^\frac{1}{2} }{n^\frac{1}{2}}   dr
\le \int_0^T  2C (\E^{t,x} ((1+ |X_r|)^{2p}))^\frac{1}{2}    \frac{\E^{t,x} (|X_r|)^\frac{1}{2} }{n^\frac{1}{2}}   dr. 
\end{align*}
By Remark \ref{RMoments}, for all $p \geq 1$ 
\begin{align}
\sup_{0\leq t\leq s \leq T} \E^{t,x} (|X_s |^p) \leq C (1 + |x|^p). 
\end{align}
At this point
the partial result
$ \sup_{(t,x) \in |0,T] \times  K} |v_n(t,x) - v(t,x) | \ra 0 $  is established.

The sequence $(h_n)$ and $(v_n)$ are compatible with 2. of Definition \ref{strongNU}, which concludes the step 1.

2)\underline{Regularization}

We have now reduced the problem to the case when $h$ has compact support. We define $h_n$ as 
\begin{equation}\label{Hn}
  h_n(r,y) = \int_{\R^d}  h(r, \xi) n \phi(n(y-\xi)) d \xi
=  \int_{\R^d} \phi(\xi) h(r,y - \frac{\xi}{n})  d \xi,
\end{equation}
where $\phi$ is a non-negative  mollifier on $\R^d $ with compact support.
$h_n$ is smooth in space with all the space derivatives being bounded.
In particular, each $h_n$ is H\"older continuous in space uniformly time.
By Lemma \ref{StrictEx} there exists a quasi-strict solution
(therefore continuous) $u = v_n$ to the problem \eqref{parabolic} with $h$ replaced by $h_n$ and again $g = 0$.
Since a quasi-strict solution is a mild solution by Proposition \ref{StrMild}, we can represent it as in \eqref{B2},
with $h_n$ as in \eqref{Hn}.
%We denote by $P(t,s;x, dz)$ the (marginal) law of $X_s$
%under $\P^{t,x}$.

It is now possible to write
\begin{align}
  |v_n(t,x) - v(t,x) | & \leq  \E^{t,x}
                         \left|  \int_{\R^d}  d \xi \  \phi(\xi)  \int_t^T
                         \left(  h(r, X_r) - h(r, X_r - \frac{\xi}{n})\right)
                         dr   \right| \\
                       & \leq    \int_{\R^d}  d \xi \  \phi(\xi)
                         \int_t^T dr \int_{\R^d}
                    \left| h(r,a) - h(r,  a - \frac{\xi}{n})  \right| P(t,r; x, da) \\
                       & \leq \int_{\R^d}  d \xi \  \phi(\xi)   \int_t^T dr \ \gamma \left(h(r, \cdot) ; \frac{diam\ (supp \ \phi)}{n}\right)
 \int_{\R^d}  P(t,r ; x, da) \\
 & \leq  \sup_{r\in [t,T]}  \gamma \left(h(r,\cdot) ; \frac{diam\ (supp \ \phi)}{n}\right) \cdot  \int_{\R^d}  d \xi \  \phi(\xi)   \int_t^T dr \int_{\R^d} P(t,r; x, da)  \\
& \leq  T \  \sup_{r\in [t,T]}  \gamma \left(h(r) ; \frac{diam\ (supp \ \phi)}{n}\right)
\end{align}
and this converges to zero by assumption, uniformly in $t \in [0,T]$. We have now proved that the sequence $v_n$ converges to $v$ uniformly on each compact. The convergence of $(h_n)$ is even stronger than required in item 2. of 
% similarly $v_n^0$ converges to $v^0$,
Definition \ref{strongNU}. Finally $v$ is a quasi-strong solution to \eqref{parabolic}. This completes the proof.  

\end{proof}

%%%%% FRA       . Riprendere da qui
%%% Riflettere eventualmente cosa succede nel caso degenere
%%% Magari le tecniche si adattano.

\section{The $C^{0,1}$-type chain rules for weak Dirichlet processes and semimartingales}

\label{sec:C01Chain}

%The following result is well-known when $u\in C^{1,2}$.
When $u \in C^{0,1}$ is a strong solution of \eqref{parabolic}, it was the object of
Theorem 4.5 of \cite{rg2}.

\begin{thm}\label{Rep} %Let $t \in [0,T[$.
Let $b: [0,T] \times \R^d \rightarrow \R^d $    and      $\sigma : [0,T]\times \R^d \rightarrow L(\R^m , \R^d) $
be Borel functions
% continuous with respect to the space variable,
and locally bounded.

Let $t \in [0,T]$.
Let $h:[t,T] \times \R^d \rightarrow \R$ and $g: \R^d \rightarrow \R$
as in the lines before \eqref{parabolic}.
Consider $u: [t,T] \times \R^d \rightarrow \R$ supposed to be of
class $C^{0,1}([t,T]\times \R^d)   $ be a quasi-strong solution of the Cauchy problem \eqref{parabolic}.

% Fix $ x\in \R^d$ and
Let $(S_s)_{s\in [t,T]}$ be a process of the form
\begin{align}
S_s = S_0 + \int_t^s \sigma(r,S_r)  dW_r  + A_s, 
\end{align}
where $(A_s)_{s\in [t,T]}$ is an $(\shf_s)_{s\in [t,T]}$-martingale orthogonal process
%having all its mutual covariations.
such $[A,A]$ exists.

Let us suppose that the assumptions below are verified.
\begin{enumerate}
\item $\int_t^s \partial_x u(r,S_r)  d^-A_r, \ s \in [t,T]$  exists
   and it is a martingale orthogonal process.
  \item \begin{align}\label{14}
&\lim_{n\rightarrow \infty } \int_t^\cdot (\partial_x u_n(r,S_r)- \partial_x u(r,S_r))  d^-A_r \\
&\qquad - \lim_{n\rightarrow \infty } \int_t^\cdot (\partial_x u_n(r,S_r)- \partial_x u(r,S_r)) \cdot  b(r,S_r) dr = 0 \qquad s\in [t,T]  \ \text{u.c.p.},
\end{align}
where the sequence $u_n$ is the {\it approximating sequence} of the
quasi-strong solution $u$. 
\end{enumerate}

Then
\begin{equation} \label{15}
  u(s,S_s)= u(t,S_t)  +  \int_t^s (\partial_x u)(r,S_r)\sigma (r,S_r) dW_r + \shb^S(u)_s,
  %-\shb^S(u)_t, 
\end{equation}
where, for $s\in [t,T]$,
\begin{equation}\label{15Bis}
 \shb^S(u)_s= \int_t^s  h(r,S_r)dr + \int_t^s \partial_x u(r,S_r)  d^-A_r - \int_t^s \partial_{x} u(r,S_r) \cdot b(r,S_r) dr,
\end{equation}
and $\shb^S(u)$ is a martingale orthogonal process. 
\end{thm}

\begin{rem} \label{RForwd}
 $\int_t^\cdot \partial_x u(r,S_r)  d^-A_r $  exists
 and it is a martingale orthogonal process in the following two cases.
 \begin{enumerate}
 \item  $A$ is a bounded variation process so that
   $S$ is a semimartingale.
   In this case $\int_t^\cdot \partial_x u(r,S_r)  d^-A_r$
   is the (componentwise) Lebesgue-Stieltjes integral
   $\int_t^\cdot \partial_x u(r,S_r)dA_r,$
   which is in particular a bounded variation process
   and therefore a martingale orthogonal process.
%In this case since the forward integral
%equals It\^o's integral.
 \item  $u \in C_{ac}^{0,2}([t,T], \R^d)$.
  Indeed, by Proposition \ref{ItoGen} and Remark \ref{Ito=Forw}
   \begin{align} \label{EForward}
%%%  Aggiusta per favore la questione del label che non appare
     % \begin{array}{ccc}
     \int_t^s \partial_x u(r,S_r)  d^-A_r &=
 u(s,S_s) - u(t,x) -  \int_t^s  (\partial_x u)(r,S_r) \sigma(r,S_r) dW_r\\
         & \qquad - \int_t^s \partial_r u(r,S_r) dr - \frac{1}{2} \int_t^s \partial^2_{xx} u(r,S_r)d[S,S]_r, \ s \in [t,T].
    % \end{array}                               
   \end{align}
      In particular the aforementioned forward integral exists.
      It remains to show that it is an
      $(\shf_s)_{s \in [t,T]}$-martingale orthogonal process.
      Let us write $M := \int_0^\cdot \sigma(r,X_r) dW_r$.
      For notational simplicity, we set $ d= 1$ and $t=  0$.
      Since $u \in C^{0,1}$, Proposition 15.3 of \cite{Russo_Vallois_Book},
      $ u(\cdot,S) = M^u + A^u,$
      where
      $M^u= \int_0^\cdot \partial_x u(r,S_r) dM_r $ 
      and $A^u$ is an    $(\shf_s)_{s \in [0,T]}$-martingale orthogonal process.
      Let $N$ be a continuous local martingale.
      We get
      \begin{equation}\label{eq:stab}
        [u(\cdot,S),N] = \int_0^\cdot \partial_x u(r,S_r) d[M,N]_r
        =  \int_0^\cdot \partial_x u(r,S_r) \sigma(r,X_r) d[W,N]_r.
 \end{equation}
 The process
 $$ u(\cdot,S) - \int_0^\cdot  \partial_x u(r,S_r) \sigma(r,S_r) dW_r, \
 t \in [0,T],$$
 is martingale orthogonal since
\begin{equation}
 \left[\int_0^\cdot  \partial_x u(r,S_r) \sigma(r,S_r) dW_r, N\right]
=  \int_0^\cdot \partial_x u(r,S_r) \sigma(r,X_r) d[W,N]_r.
 \end{equation}
 Finally  the covariation of the right-hand side of \eqref{EForward}
with $N$ vanishes since
all the other processes appear in the right-hand
side of \eqref{EForward} are martingale orthogonal being of bounded
variation.

\end{enumerate}
\end{rem}
\begin{rem} \label{rmk:DerConv}
  In fact condition \eqref{14} is a bit obscure: it is for instance verified
  if $A$ has bounded variation and if there is a quasi-strong solution
 such that space derivatives of the approximating sequence $u_n$, i.e. 
$\partial_x u_n$, converges pointwise (with uniform bound in $n$
       on each compact)
       to the derivative of $u$.
       
       This can happen even in a possibly degenerate case,
if for instance,       
for all $s \in [0,T]$,
       $\sigma(s, \cdot), b(s,\cdot)$ are of class $C^1$
       such that $\partial_x \sigma, \partial_x b$ are uniformly bounded.
       We also suppose $\sigma, b: [t,T] \times \R^d \rightarrow \R$ a.e. continuous  in time for all $x \in \R^d$.
       Concerning the other coefficients we suppose $g$ (resp. $h$)
       such that $\partial_x g$ (resp. $\partial_x h$)
       has polynomial growth.
   In this case \eqref{parabolic} admits a mild $C^{0,1}$-solution
       of the type
       $$ u(s,x) = \E(g(Y^{s,x}_T) + \int_t^T h(r,Y_T^{r,x}) dr),
       (s,x) \in [t,T]\times \R^d,$$
       where $Y^{s,x}$ is a strong solution of the SDE 
       \eqref{A1}
       on some probability space $(\Omega, \shf,\P)$
       with related expectation $\E$.
       In this case, adopting the formalism with $d=1$ for simplicity
\begin{equation} \label{eq:derivative}
       \partial_x u(s,x) = \E(g'(Y^{s,x}_T) Z^{s,x}_T ) +
       \int_t^T \partial_x h(r,Y_T^{r,x}) Z^{r,x}_T  dr),
       \end{equation}
       where $Z^{s,x}:= \partial_x Y^{s,x}$,
       the derivative being taken in $L^2(\Omega)$.

       By the same construction as in the proof of
       Proposition \ref{P321} (truncation-regularization procedure)
       $u$ can be shown to be a quasi-strong solution
       of \eqref{parabolic} by approximation via solutions $u_n$,
       where $u_n$ is a solution of \eqref{parabolic} replacing
       $g$ (resp. $h$) by $g_n$ (resp. $h_n$)
       by a truncation-regularization procedure.
       We apply \eqref{eq:derivative} replacing
       $g, h, u$ with $g_n, h_n, u_n$.
       In this way $\partial_x u_n$ suitably converges to $\partial_x u$.

    \end{rem}
 
\begin{proof}[Proof of Theorem \ref{Rep}]
  For simplicity of the formulation we set $d=1$ and without restriction of generality we set $t=0$.
  % We will basically apply Proposition \ref{FukuDir} and for this Assumption \eqref{14} is crucial.
  Let $u_n \ra u  $ be a sequence as in Definition \ref{strongNU}. By Proposition \ref{ItoGen}
and Remark \ref{Ito=Forw},
  we get 
  \begin{align*}
    %\label{BB}
    u_n(s, S_s)&= u_n(0,S_0) + \int_0^s \shl_0 u_n(r,S_r) dr - \int_0^s \partial_x u_n(r,S_r) b(r,S_r)  dr \\
                 %\nonumber \\
           &\qquad
       + \int_0^s \partial_x u_n (r,S_r) d^- A_r + N^n_s,
\end{align*}
where  
\begin{align}
  N^n_s  \doteqdot
  \int_0^s \partial_x u_n(r,S_r) \sigma(r,S_r) dW_r ,
\end{align}
which is in particular
a local martingale vanishing at zero.
By assumption \eqref{14}, the process $N^n$ converges u.c.p. to
\begin{equation}\label{eq:mart}
 N_s
    \doteqdot
 u(s, S_s) - u(0,S_0) - \bar {\shb}^S(u)_s,
  \end{equation}
where
\begin{equation}
  %\label{15Ter}
 {\bar \shb}^S(u)_s= \int_0^s  h(r,S_r)dr + \int_0^s \partial_x u(r,S_r)  d^-A_r - \int_0^s \partial_{x} u(r,S_r)   b(r,S_r) dr,
\end{equation}
%   \int_0^s h(r,S_r) dr + \int_0^s  \partial_x u(r,S_r) b(r,S_r)  dr -  \int_0^s \partial_x u (r,S_r) d^- A_r.
% \end{equation}

% \begin{equation}\label{BBB}
% N_s \doteqdot  u(s, S_s) - u(0,S_0) - \int_0^s h(r,S_r) dr + \int_0^s  \partial_x u(r,S_r) b(r,S_r)  dr -  \int_0^s \partial_x u (r,S_r) d^- A_r.
% \end{equation}
% At this point we have the decomposition
% \begin{align}
% u(s,S_s) = u(0,S_0) + N_s + {\shb}^S(u)_s, \qquad s\in [0,T],
% \end{align}
%where ${\shb}^S(u)$ is defined by the right-hand side of  \eqref{15Bis}.
Now,  the  space of the space of continuous  local martingales vanishing at zero,
  %It is straightforward to observe that $\shm^c$ and $\shm^c_{loc}$ are
as a linear subspace of $ C_\shf([0,T] \times \Omega ; \R), $
 is closed (under the u.c.p. convergence topology),
  see e.g.  Proposition 4.4 in \cite{rg2}.
 Consequently $N$ is a continuous local martingale vanishing at zero.
It remains to prove the following.
\begin{enumerate}
  \item ${\bar \shb}^S(u)$ is a martingale orthogonal process.
\item
\begin{equation} \label{eq:N}
  N_s =  \int_0^s \partial_x u(r,S_r) \sigma(r,S_r) dW_r.
\end{equation}

\end{enumerate}
 1. follows by additivity since 
 $\int_0^\cdot \partial_x u(r,S_r)d^-A_r$
and all bounded variation processes
 are martingale orthogonal
processes.

Concerning 2.,
setting  $M_s =   \int_0^s \sigma(r,S_r) dW_r,  $  
besides
\eqref{eq:mart},
Proposition  \ref{FukuDir} 1. provides another weak Dirichlet decomposition 
\begin{align}
u(s, S_s) = u(0, S_0) +  \int_0^s \partial_x u(r,S_r) \sigma(r,S_r) dW_r  + \shb^S(u)_s,
\end{align}
where $\shb^S(u)$ is a martingale orthogonal process. 
By the uniqueness of the weak Dirichlet decomposition we get
finally \eqref{eq:N}.
% and 
% \begin{align} \label{BBBar}
%  \bar{\shb}^S(u)_s = \int_0^s  h(r,S_r)dr + \int_0^s \partial_x u(r,S_r)d^-A_r - \int_0^s \partial_{x} u(r,S_r) b(r,S_r)dr = \shb^S(u)_s.
% \end{align}
%and therefore $\bar{\shb}^S(u)_s$ is a martingale orthogonal process. 
The proof is now complete.

\end{proof}
%{\bf FINE VERIFICHE.}

In the next corollary we will apply Theorem \ref{Rep} to a solution $S$ of a 
non-Markovian  SDE.
That result extends Corollary 4.6 of \cite{rg2},
where the coefficients $\sigma, b$ were continuous.
Here, we also assume $u \in C^{0,1}([t,T[\times \R^d) \cap C^{0}([t,T] \times \R^d)$, that is,
$u(T,\cdot)$ may not be differentiable in the space variable.
This allows lower regularity in the terminal condition of the
PDE \eqref{parabolic}.

%\begin{cor}\label{representation}
\begin{cor}\label{representation}
  Let $f: \Omega \times [0,T] \times \R^d \rightarrow \R^d$ be a
  progressively measurable field
  % , continuous in $x$
  and a.s. locally bounded.
 % i.e. for all $K$ compact, $\sup_{r \in [t,T], x\in K} |f(r,x)| < \infty$  a.s.
  Let $b: [0,T] \times \R^d \rightarrow \R^d $  and
  $\sigma : [0,T]\times \R^d \rightarrow L(\R^m , \R^d) $ be  Borel locally bounded functions.
 % continuous with respect to $x$
Let $(S_s)_{s\in [t,T]}$ be a solution to the SDE
\begin{align}
dS_s = f(s,S_s) ds + \sigma(s,S_s)dW_s.
\end{align}
Let $t \in [0,T]$,
 $h:[t,T] \times \R^d \rightarrow \R$ and $g: \R^d \rightarrow \R$
as in the lines before \eqref{parabolic}.
Let  $u  \in C^{0,1}([t,T[\times \R^d) \cap C^{0}([t,T] \times \R^d)$ be a quasi-strong solution of the Cauchy problem \eqref{parabolic}
  fulfilling
\begin{equation}\label{L21} 
%\int_t^T \sup_{x \in K} |\partial_x u(s,x)|^2 ds < \infty.
\int_0^T \vert \partial_x u(s,S_s) \vert^2 ds  < \infty  \ {\textit a.s.}
\end{equation}
% Fix $ x\in \R^d$ 
% Let $(S_s)_{s\in [t,T]}$ be a solution to the SDE
% \begin{align}
% dS_s = f(s,S_s) ds + \sigma(s,S_s)dW_s,
% \end{align}
%with initial condition $S_t = x \in \R^d$.
Moreover we suppose the validity of one of the following items.
\begin{enumerate} 
\item The approximating sequence $u_n$ of Definition \ref{strongNU}
  fulfills
  \begin{align}\label{16}
 \lim_{n\rightarrow \infty}  \int_t^\cdot  (\partial_x u_n(r,S_r)- 
\partial_x u(r,S_r)) \cdot (f(r,S_r) - b(r,S_r)) dr = 0 \ \ \text{u.c.p.}
\end{align}
\item $\sigma \sigma^\top(r,S_r)$ is invertible
  for every $r \in [t,T]$ and denote 
  the right pseudo-inverse $ \sigma^{-1}(r,S_r):=
  \sigma^{\top} (\sigma \sigma^{-1})(r,S_r).$
  Moreover the Novikov condition  
\begin{equation} \label{Novikov}
  \E\left ( \exp\left(\frac{1}{2} \int_t^T
 \vert \sigma^{-1}(f-b)\vert^2(r,S_r) dr\right)\right) < \infty,
  \end{equation}
  holds.
  %where $\sigma^{-1}$ stands for the pseudo inverse of $\sigma$.
\end{enumerate}
Then, for $s\in [t,T],$
\begin{equation}\label{215}
  u(s,S_s)= u(t,S_t)  +  \int_t^s  (\partial_x u) (r,S_r)\sigma (r,S_r) dW_r + \shb^S(u)_s,% -\shb^S(u)_t, 
\end{equation}
where, 
\begin{equation}\label{216}
 \shb^S(u)_s   = \int_t^s  h(r,S_r)dr + \int_t^s  \partial_x u(r,S_r) \cdot f(r,S_r)  dr - \int_t^s \partial_{x} u(r,S_r) \cdot b(r,S_r) dr.
\end{equation}
\end{cor}

\begin{rem} \label{RNovikov}
Clearly, whenever $(\omega,s,x) \mapsto \sigma^{-1}(s,x) (f(\omega,s,x) - b(s,x)) $
is bounded, then \eqref{Novikov} is fulfilled.
\end{rem}
\begin{proof}[Proof of Corollary \ref{representation}.]
In the proof, we set $t=0$ and $d=1$ without loss of generality.
We first prove the result assuming that item 1. holds.
We set
\begin{align}
A = \int_0^\cdot f(r,S_r) dr. 
\end{align}
Let $\ve > 0$. The idea is to apply Theorem \ref{Rep} with $T$ replaced by $T - \ve$. The function $u$ restricted to $[0, T- \ve] \times \R^d$ is trivially a quasi-strong solution to 
\begin{equation}
  \shl_0 u(s,x) = h(s,x), \ u(T-\ve, \cdot) = g_\ve,
\end{equation}
where $g_\ve = u(T-\ve, \cdot)$.
By Theorem \ref{Rep}, taking into account Remark \ref{Ito=Forw},
we get the decomposition  \eqref{215} and \eqref{216} for $s \in [0 , T- \ve]$.
This implies those decompositions for $s \in [0,T[$.
The case $s = T$ follows letting
% $\ve \ra 0$
$ s \rightarrow T$ and
using in particular condition \eqref{L21}.

Now let us assume that item 2. is verified. 
Let $\Q$ be the probability measure on $(\Omega, \shf_T)$ defined
\begin{equation}
  d \Q  = Z_T (\eta) d \P,
%    \sigma^{-1}(s,S_s)  (b(s,S_s) - f(s,S_s(\omega)))\right)d \P,
\end{equation}
%where for any $(\shf_r)_{r \in [0,T]}$-adapted process $(\eta_r)_{r\in [0,T]}$, we define
where
 \begin{equation}
   Z_s(\eta) =  e^ {\left(\int_0^s \eta_r dW_r - \frac{1}{2} \int_0^s \left| \eta_r\right|^2 dr \right)},
   \ s \in [0,T]
 \end{equation}
 and
 \begin{equation}
  \eta_s = \sigma^{-1}(s,S_s)  (b(s,S_s) - f(s,S_s)), \ s \in [0,T].
  \end{equation}
Then, the Novikov condition implies that  $Z(\eta)$ is a martingale.
Hence by Girsanov's Theorem,
% \cite{Girsanov},
the process 
\begin{equation}
\tilde{W}_s  =  W_s - \int_0^s \sigma^{-1}(r,S_r) (b(r,S_r)- f(\omega, r,S_r) ) dr, \ s \in [0,T],
\end{equation}
is an $(\shf_s, \Q)$-Brownian motion and $(S_s)_{s\in[0,T]}$ satisfies the stochastic differential
equation
\begin{equation}
dS_s = b(s,S_s) ds + \sigma(s,S_s) d\tilde{W}_{s}.
\end{equation}
Then, by item 1. in the statement (here $\Q$ replaces $\P$, $f = b$
and $\tilde W$ replaces $W$),
we obtain
\begin{comment}
we apply Theorem \ref{Rep} (under $\Q$)  observing that condition \eqref{14} is trivially satisfied since the left-hand side is equal to zero for all $n$.
This entails 
\end{comment}
\begin{equation}
 u(s,S_s)= u(0,S_0)  +  \int_0^s \partial_x u(r,S_r)\sigma (r,S_r) d\tilde W_r + \int_0^s h(r,S_r) dr, 
\ s \in [t,T],
\end{equation}
and the conclusion follows.
\end{proof}

\begin{rem} \label{R223}
  If $\lim_{n\rightarrow \infty} \partial_ xu_ n = \partial_x u $   in $C^0 ([t,T]\times \R^d)$, then assumption \eqref{16} is trivially verified.

\end{rem}

{\bf ACKNOWLEDGMENTS.}

The authors would like to thank the anonymous Referee
for having carefully read the paper, which  has
permitted us to improve its quality.
The research of the second named author
was partially supported by the  ANR-22-CE40-0015-01 project (SDAIM).

%\newpage
\addcontentsline{toc}{chapter}{Bibliography}
\bibliographystyle{plain}
\bibliography{../../../BIBLIO_FILE/biblioCarlo}

\end{document}